\documentclass[12pt]{article}
\usepackage{amsmath,amssymb,amsthm,amsfonts,amsbsy, enumerate}
\usepackage[dvips]{epsfig}
\usepackage{listings}
\theoremstyle{plain}
\newtheorem{theorem}{Theorem}[section]
\newtheorem{corollary}[theorem]{Corollary}

\newtheorem{lemma}[theorem]{Lemma}
\newtheorem{prop}[theorem]{Proposition}

\begin{document}

\title{Disconnecting strongly regular graphs}
\author{Sebastian M. Cioab\u{a} \footnote{Department of Mathematical Sciences,  University of Delaware, Newark, DE 19707-2553, USA. {\tt cioaba@math.udel.edu}. This work was partially supported by a grant from the Simons Foundation ($\#209309$ to Sebastian M. Cioab\u{a}) and by National Security Agency grant H98230-13-1-0267.} \,, Jack Koolen  \footnote{School of Mathematical Sciences, University of Science and Technology of China, 96 Jinzhai Road, Hefei,
 230026, Anhui, P.R. China. {\tt koolen@ustc.edu.cn}. This research was partially supported by a grant of the `One hundred talent program' of the Chinese government. } \, and Weiqiang Li\footnote{Department of Mathematical Sciences,  University of Delaware, Newark, DE 19707-2553, USA. {\tt weiqiang@udel.edu}}  }
\date{November 22, 2013}
\maketitle

\begin{abstract}
In this paper, we show that the minimum number of vertices whose removal disconnects a connected strongly regular graph into non-singleton components, equals the size of the neighborhood of an edge for many graphs. These include blocks graphs of Steiner $2$-designs, many Latin square graphs and strongly regular graphs whose intersection parameters are at most a quarter of their valency. 
\end{abstract}

\section{Introduction}

A graph $G$ is strongly regular with parameters $v, k, \lambda$ and $\mu$ (shorthanded $(v,k,\lambda,\mu)$-SRG from now on) if it has $v$ vertices, is $k$-regular, any two adjacent vertices have exactly $\lambda$ common neighbors and any two non-adjacent vertices have exactly $\mu$ common neighbors. The study of strongly regular graphs lies at the intersection of graph theory, algebra and finite geometry \cite{BL,Cam,Cam2} and has applications in coding theory and computer science, among others \cite{CamLint,Spielman}.

In this paper, we study the minimum size of a subset of vertices of a connected strongly regular graph whose removal disconnects the graph into non-singleton components. In 1985, Brouwer and Mesner (see \cite{BM} or \cite[Section 9.3]{BH}) used eigenvalue interlacing and Seidel's characterization of strongly regular graphs with minimum eigenvalue $-2$ (see \cite{S} or \cite[Section 9.2]{BH}) to prove that the vertex-connectivity of any connected strongly regular graph equals its valency. Brouwer and Mesner also showed that any disconnecting set of minimum size must be the neighborhood of some vertex. Brouwer and Koolen \cite{BK} proved the same results for distance-regular graphs. An important conjecture of Brouwer \cite{Br1} is to extend these results to any connected graph that is a color class of an association scheme.

In 1996, Brouwer \cite{Br1} conjectured that the minimum size of a disconnecting set of vertices whose removal disconnects a connected $(v,k,\lambda,\mu)$-SRG into non-singleton components equals $2k-\lambda-2$, which is the size of the neighborhood of an edge. Cioab\u{a}, Kim and Koolen \cite{CKK} showed that there are strongly regular graphs for which the above statement does not hold. However, it seems that for many families of strongly regular graphs, Brouwer's Conjecture is true. In this paper, we extend several results from \cite{CKK} and we show that Brouwer's Conjecture is true for any $(v,k,\lambda,\mu)$-SRG with $\max(\lambda,\mu)\leq k/4$. This makes significant progress towards solving an open problem from \cite{CKK} stating that Brouwer's Conjecture is true for any $(v,k,\lambda,\mu)$-SRG with $\lambda<k/2$. We also prove that Brouwer's Conjecture is true for any block graph of a Steiner $2$-$(n,K,1)$-design when $K\in \{3,4\}$ and for any Latin square graph with parameters $(n^2,t(n-1),n-2+(t-1)(t-2),t(t-1))$, when $n\geq 2t\geq 6$. Our results and Neumaier's characterization of strongly regular graphs with fixed minimum eigenvalue \cite{N} enable us to verify the status of Brouwer's Conjecture for all but finitely many strongly regular graphs with minimum eigenvalue $-3$ or $-4$. We also prove that the edge version of Brouwer's Conjecture is true for any connected strongly regular graph; we show that the minimum number of edges whose removal disconnects a $(v,k,\lambda,\mu)$-SRG into non-singletons, equals $2k-2$, which is the edge-neighborhood of an edge.

Our graph theoretic notation is standard (for undefined notions see \cite{BH,GR}). The adjacency matrix of a graph $G$ has its rows and columns indexed after the vertices of the graph and its $(u,v)$-th entry equals $1$ if $u$ and $v$ are adjacent and $0$ otherwise. If $G$ is a connected $k$-regular graph of order $v$, then $k$ is the largest eigenvalue of the adjacency matrix of $G$ and its multiplicity is $1$ (see \cite{BCN,BH,GR}). In this case, let $k=\theta_1>\theta_2\geq \dots \geq \theta_v$ denote the eigenvalues of the adjacency matrix of $G$. If $G$ is a connected $(v,k,\lambda,\mu)$-SRG, then $G$ has exactly three distinct eigenvalues; let $k>\theta_2>\theta_v$ be the distinct eigenvalues of $G$, where $\theta_2=\frac{\lambda-\mu+\sqrt{(\lambda-\mu)^2+4(k-\mu)}}{2}$ and $\theta_v=\frac{\lambda-\mu-\sqrt{(\lambda-\mu)^2+4(k-\mu)}}{2}$ (see \cite{BCN,BH,GR} for more details). Thus, $\theta_2+\theta_v=\lambda-\mu$ and $\theta_2\theta_v=\mu-k$. If $X$ is a subset of vertices of a graph $G$, let $N(X)=\{y\notin X: y\sim x \text{ for some } x \in X\}$ denote the neighborhood of $X$. If $G$ is a $(v,k,\lambda,\mu)$-SRG, then $|N(\{u,v\})|=2k-\lambda-2$ for every edge $uv$ of $G$. We denote by $\kappa_2(G)$ the minimum size of a disconnecting set of $G$ whose removal disconnects the graph into non-singleton components if such a set exists. This parameter has been studied for many families of graphs (see Boesch and Tindell \cite{BT}, Balbuena, Carmona, F\`{a}brega and Fiol \cite{BCFF}, F\`{a}brega and Fiol \cite{FF1,FF2}) or Hamidoune, Llad\'{o} and Serra \cite{HLS} for example). Let $G$ be a connected $(v,k,\lambda,\mu)$-SRG. We say that $G$ is OK if either it has no disconnecting set such that each component has as at least two vertices, or if $\kappa_2(G) = 2k-\lambda-2$.

Let $G$ be a connected graph. If $S$ is a disconnecting set of $G$ of minimum size such that the components of $G\setminus S$ are not singletons, then denote by $A$ the vertex set of one of the components of $G\setminus S$ of minimum size. By our choice of $A$, $|B|\geq |A|$, where $B:=V(G)\setminus (A\cup S)$. As $S$ is a disconnecting set, $N(A)\subset S$ and consequently, $|S|\geq |N(A)|$. Note that it is possible for the disconnecting set $S$ to contain a vertex $y$ and its neighborhood $N(y)$ in which case $y\in S$, but $y\notin N(A)$ and thus, $S\neq N(A)$. In order to prove Brouwer's Conjecture is true for a $(v,k,\lambda,\mu)$-SRG $G$ with vertex set $V$ and $v\geq 2k-\lambda+3$, we will show that $|S|\geq 2k-\lambda-2$ for any subset of vertices $A$ with $3\leq |A|\leq \frac{v}{2}$ having the property that $A$ induces a connected subgraph of $G$. In some situations, we will be able to prove the stronger statement that $|N(A)|>2k-\lambda-2$. Throughout the paper, $S$ will be a disconnecting set of $G$, $A$ will stand for a subset of vertices of $G$ that induces a connected subgraph of $G\setminus S$ of smallest order and $B:=V(G)\setminus (A\cup S)$. As before, $N(A)\subset S$ and thus, $|S|\geq |N(A)|$. Let $a=|A|, b=|B|$ and $s=|S|$. We will need the following results. 

\begin{lemma}[Lemma 2.3 \cite{CKK}]\label{Haemers} If $G$ is a connected $(v,k,\lambda,\mu)$-SRG, then
\begin{equation}\label{Eq1}
 |S|\geq \frac{4ab\mu}{(\lambda-\mu)^2+4(k-\mu)}.
\end{equation}
\end{lemma}


\begin{prop}\label{2kl2}
Let $G$ be a connected $(v,k,\lambda,\mu)$-SRG and $c\geq 3$ a fixed integer. If $a\geq c$ and 
\begin{equation}\label{eq2kl2}
4(c-2)\left[(k-\mu)\left(k+c-1-\frac{\lambda(c-1)}{c-2}\right)-ck\right]>(\lambda-\mu)^2(2k-\lambda-2)
\end{equation}
then $|S|>2k-\lambda-2$.
\end{prop}
\begin{proof}
Let $s$ denote the minimum size of a disconnecting set $S$ whose removal leaves only 
non-singleton components. Assume that $s\leq 2k-\lambda-2$. This implies $a+b=v-s\geq
v-(2k-\lambda-2)=v+\lambda+2-2k$. As $b\geq a\geq c$, we obtain $ab\geq
 c(v+\lambda+2-2k-c)$. This inequality,  Lemma \ref{Haemers} and $v=1+k+k(k-\lambda-1)/\mu$ imply
\begin{align*}
s&\geq \frac{4ab\mu}{(\lambda-\mu)^2+4(k-\mu)}\geq\frac{4c(v+\lambda+2-2k-c)\mu}{(\lambda-\mu)^2+4(k-\mu)}\\
&=\frac{4c[k^2+(-\lambda-\mu-1)k+\mu(\lambda+3-c)]}{(\lambda-\mu)^2+4(k-\mu)}\\
&>2k-\lambda-2,
\end{align*}
where the last inequality can be shown that is equivalent to our hypothesis \eqref{eq2kl2} by a straightforward calculation. This contradiction finishes our proof.
\end{proof}

\begin{lemma}\label{Clique}
Let $G$ be a connected $(v,k,\lambda,\mu)$-SRG and let $C$ be a clique with $q\geq 3$ vertices contained in $G$. If $k-2\lambda-1>0$, then $|N(C)|>2k-\lambda-2$.
\end{lemma}
\begin{proof}
If $x, y$ and $z$ are three distinct vertices of $C$, then $x, y$ and $z$ have at least $q-3$ common neighbors. Thus,  by inclusion and exclusion, we obtain that
\begin{align*}
|N(C)|&\geq |N(\{x,y,z\})|-(q-3)\\
&\geq 3(k-2)-3(\lambda-1)+(q-3)-(q-3)\\
&=2k-\lambda-2+(k-2\lambda-1)\\
&>2k-\lambda-2.
\end{align*}
\end{proof}

We will also need the following {\em inproduct bound} (see \cite[Lemma 2.2]{BK} for a short proof).
\begin{lemma}[Brouwer and Koolen \cite{BK}]\label{codingtheory}
Among a set of $a$ binary vectors of length $n$ and average weight $w$, there are two with inner product at least $w\left(aw/n-1\right)/(a-1)=\frac{w^2}{n}-\frac{w(n-w)}{n(a-1)}$.
\end{lemma}

\section{Brouwer's Conjecture is true when $\max(\lambda,\mu)\leq k/4$}

In this section, we prove that Brouwer's Conjecture is true for all connected $(v,k,\lambda,\mu)$-SRGs with $\lambda$ and $\mu$ relatively small.
\begin{theorem}\label{k/4}
If $G$ is a connected $(v,k,\lambda,\mu)$-SRG with $\max(\lambda,\mu)\leq k/4$, then $\kappa_2(G)=2k-\lambda-2$.
\end{theorem}
\begin{proof}
Let $G$ be a $(v,k,\lambda,\mu)$-SRG with $\max(\lambda,\mu)\leq k/4$. Assume that $S$ is a disconnecting set of vertices of size $s=|S|\leq 2k-\lambda-3$ such that $V(G)\setminus S=A\cup B$ and there are no edges between $A$ and $B$.  Assume that each component of $A$ and $B$ has at least $3$ vertices from now on. Let $s=|S|, a=|A|\geq 3, b=|B|\geq 3$.

We first prove that $b+s\geq a+s\geq 9k/4$. Assume that $a+s<9k/4$. As each component of $A$ has at least $3$ vertices, it means we can find three vertices $x,y,z$ in $A$ that induce a triangle or a path of length $2$. If $x,y,z$ induce a triangle, then $|N(\{x,y,z\})|\geq 3(k-2)-3(\lambda-1)=3k-3\lambda-3$. If $x,y,z$ induce a path of length $2$, then $|N(\{x,y,z\})|\geq 3k-4-(2\lambda+\mu-1)=3k-2\lambda-\mu-3$. In either case, we obtain that $|N(\{x,y,z\})|\geq 3k-3\max(\lambda,\mu)-3$. As $\{x,y,z\}\cup N(\{x,y,z\})\subset A\cup S$, we deduce that $9k/4>a+s\geq 3k-3\max(\lambda,\mu)$ which implies $\max(\lambda,\mu)>k/4$ contradicting our hypothesis. Thus, $b+s\geq a+s\geq 9k/4$.

For two disjoint subsets of vertices $X$ and $Y$, denote by $e(X,Y)$ the number of edges between $X$ and $Y$. Let $\theta_1(X)$ be the largest eigenvalue of the adjacency matrix of the subgraph induced by $X$. Denote by $X^c$ the complement of $X$. Let $\alpha=\frac{e(A,A^c)}{a}=\frac{e(A,S)}{a}$ and $\beta=\frac{e(B,B^c)}{b}=\frac{e(B,S)}{b}$. We consider two cases depending on the values of $\alpha$ and $\beta$.

{\bf Case 1.} $\max(\alpha,\beta)\leq 3k/4$.

We first prove that $\theta_2\geq k/4$. Since $\alpha\leq 3k/4$, the average degree of the subgraph induced by $A$ is $k-\alpha\geq k/4$. Similarly, as $\beta\leq 3k/4$, the average degree of the subgraph induced by $B$ is $k-\beta\geq k/4$. Eigenvalue interlacing (see \cite[Section 2.5]{BH}) implies $\theta_2\geq \min(\theta_1(A),\theta_1(B))\geq \min(k-\alpha,k-\beta)\geq k/4$. As $\theta_2\theta_v=\mu-k>-k$ and $\theta_2\geq k/4$, we deduce that $\theta_v>-4$. This implies $G$ is a conference graph (for which we know the Brouwer's Conjecture is true as proved in \cite{CKK}) or $\theta_v\geq -3$. The case $\theta_v=-2$ was solved completely in \cite{CKK} so we may assume that $\theta_v=-3$. As $\lambda-\mu=\theta_2+\theta_v$, we get $\lambda-\mu\geq k/4-3$. If $\mu\geq 4$, then $\lambda\geq k/4+1$, contradiction. Thus, we may assume $1\leq \mu\leq 3$. Because $\theta_n=-3$, we get $\theta_2=\frac{k-\mu}{3}$. This implies $\lambda=\mu+\theta_2+\theta_n=\mu+\frac{k-\mu}{3}-3=\frac{k+2\mu-9}{3}$. 

If $\mu=1$, then $\lambda=\frac{k-7}{3}>\frac{k}{4}$ when $k>28$. If $k\leq 28$ and $k\equiv 1\pmod 3$, the possible parameter sets are: $(50,7,0,1), (91,10,1,1), (144,13,2,1), (209,16,3,1), (286,19,4,1), \\(375,22,5,1), (476,25,6,1), (589,28,7,1)$. The only parameter sets with integer eigenvalue multiplicities are $(50,7,0,1), (209,16,3,1)$ and $(375,22,5,1)$. A strongly regular graph with $\mu=1$ must satisfy the inequality $k\geq (\lambda+1)(\lambda+2)$ (see \cite{Bag,BN}). This implies there are no strongly regular graphs with parameters $(209,16,3,1)$ and $(375,22,5,1)$. The parameters $(50,7,0,1)$ correspond to the Hoffman-Singleton graph which is OK as proved in \cite{CKK}. 

If $\mu=2$, then $\lambda=\frac{k-5}{3}>\frac{k}{4}$ when $k>20$. If $k\leq 20$ and $k\equiv 2\pmod 3$, the feasible parameter sets are $(16,5,0,2), (33,8,1,2), (56,11,2,2), (85,14,3,2), (120,17,4,2),\\ (161,20,5,2)$. The only parameter sets with integer eigenvalue multiplicities are $(16,5,0,2)$ and $(85,14,3,2)$. The parameter set $(16,5,0,2)$ corresponds to the Clebsch graph which is OK as proved in \cite{CKK}. It is not known whether there exists a strongly regular graphs with the parameters $(85,14,3,2)$. However, this parameter set satisfies Lemma 2.5 from \cite{CKK} so if it exists, such a graph is OK. 

If $\mu=3$, then $\lambda=\frac{k-3}{3}>\frac{k}{4}$ when $k>12$. If $k\leq 12$ and $k\equiv 0\pmod 3$, the feasible parameter sets are $(6,3,0,3), (15,6,1,3), (28,9,2,3)$ and $(45,12,3,3)$. A $(6,3,0,3)$-SRG is isomorphic to $K_{3,3}$ and a $(15,6,1,3)$-SRG is isomorphic to the complement of the triangular graph $T(6)$. By \cite[Proposition 10.2]{CKK}, both these graphs are OK. The other parameter set with integer eigenvalue multiplicities is $(45,12,3,3)$. There are exactly $78$ strongly regular graphs with parameters $(45,12,3,3)$ (see \cite{CDS}) and they are all OK by \cite[Lemma 2.5]{CKK}.

{\bf Case 2.} $\max(\alpha,\beta)>3k/4$.

Assume $\alpha>3k/4$; the case $\beta>3k/4$ is similar (replace $A$ by $B$, $a$ by $b$ and $\alpha$ by $\beta$ in the analysis below) and will be omitted. Applying Lemma \ref{codingtheory} to the characteristic vectors of the neighborhoods (restricted to $S$) of the vertices in $A$, we deduce that there exist two distinct vertices $u$ and $v$ in $A$ such that $|N(u)\cap N(v)\cap S|\geq \frac{\alpha \left(\frac{a\alpha}{s}-1\right)}{a-1}$. As $a+s\geq 9k/4$ and $s\leq 2k-\lambda-3$, we obtain that $a\geq k/4+\lambda+3$. Because $\alpha/s>\frac{3k}{4(2k-\lambda-3)}$, we get that $|N(u)\cap N(v)\cap S|\geq \frac{\alpha \left(\frac{a\alpha}{s}-1\right)}{a-1}>\frac{3k}{4}\cdot \frac{a\cdot \frac{3k}{4(2k-\lambda-3)}-1}{a-1}$. The right-hand side of the previous inequality is greater than $k/4$ if and only if $\frac{a\cdot \frac{3k}{4(2k-\lambda-3)}-1}{a-1}>\frac{1}{3}$. This is equivalent to $a\cdot \frac{k+4\lambda+12}{8k-4\lambda-12}>2$. Since $a\geq k/4+\lambda+3$, the previous inequality is true whenever $\left(k/4+\lambda+3\right)(k+4\lambda+12)>2(8k-4\lambda-12)$. This is the same as 
\begin{equation}\label{klambda}
\left(\frac{k+4(\lambda+3)}{2}\right)^2>16\left(k-\frac{\lambda+3}{2}\right),
\end{equation}
which is equivalent to
\begin{equation}\label{klambda2}
k+4(\lambda+3)>8\sqrt{k-\frac{\lambda+3}{2}}.
\end{equation}
For $\lambda\geq 1$, this inequality is true as 
$k+4(\lambda+3)\geq 4\sqrt{k(\lambda+3)}\geq 8\sqrt{k}>8\sqrt{k-\frac{\lambda+3}{2}}$.

For $\lambda=0$, the inequality \eqref{klambda2} is true for all $k$ except when $8\leq k\leq 32$. As $\mu\leq k/4$, this implies that $1\leq \mu\leq 8$. We show that the condition 
\begin{equation}\label{CKK25}
4(k-2\lambda)(k-\mu)>(\lambda-\mu)^2(2k-\lambda-3)
\end{equation}
of Proposition 2.4 from \cite{CKK} is satisfied in each of these cases, therefore showing that $G$ is OK and finishing our proof. The inequality above is the same as $4k(k-\mu)>\mu^2(2k-3)$. As $k-\mu\geq 3\mu$, the previous inequality is true whenever $4k(3\mu)>\mu^2(2k-3)$ which is true when $12k>\mu(2k-3)$. This last inequality is true whenever $1\leq \mu\leq 6$.


If $\mu=7$, then \eqref{CKK25} is $4k(k-7)>49(2k-3)$ which holds when $k\geq 31$. As $k\geq 4\mu=28$, the previous condition will be satisfied except when $k\in \{28,29,30\}$. But in this situation, the strongly regular graph does not exist. This is because $\theta_2+\theta_v=\lambda-\mu=-7$ and $\theta_2\theta_v=\mu-k\in \{-21,-22,-23\}$ which is impossible as $\theta_2$ and $\theta_v$ are integers. 

If $\mu=8$ and $k=32$, the graph would have parameters $(157,32,0,8)$. However, such a graph does not exist as $\theta_2$ and $\theta_v$ would have to be integers satisfying $31=k-\lambda-1=-(\theta_2+1)(\theta_v+1)$ and $\theta_2\theta_v=\mu-k=-24$ which is again impossible as $\theta_2$ and $\theta_v$ are integers. This finishes our proof.
\end{proof}

We checked the tables with feasible parameters for strongly regular graphs on Brouwer's homepage \cite{Bwebpage}. The following parameters satisfy the condition $\max(\lambda,\mu)\leq k/4$ and are not parameters of block graphs of Steiner $2$-designs or Latin square graphs with less than $200$ vertices: 
$(45,12,3,3), (50,7,0,1), (56,10,0,2),(77,16,0,4),(85,14,3,2)^?,(85,20,3,5), (96,19,2,4), 
\\(96,20,4,4), (99,14,1,2)^?,(115,18,1,3)^?, (125,28,3,7), (133,24,5,4)^?, (133,32,6,8)^?, 
\\(156,30,4,6), (162,21,0,3)^?,(162,23,4,3)^?,(165,36,3,9), (175,30,5,5), (176,25,0,4)^?, 
\\(189,48,12,12)^?, (196,39,2,9)^?$. The existence of the graphs with ``$?$" is unknown at this time. 

\section{Brouwer's Conjecture for the block graphs of Steiner $2$-designs}

A Steiner 2-$(n,K,1)$-design is a point-block incidence structure on $n$ points, such that each block has $K$ points and any two distinct points are contained in exactly one block. The block graph of such a design has as vertices the blocks of the design and two distinct blocks are adjacent if and only if they intersect. The block graph of a Steiner 2-$(n,K,1)$-design is a strongly regular graph with parameters $\left(\frac{n(n-1)}{K(K-1)}, \frac{K(n-K)}{K-1}, (K-1)^2+\frac{n-1}{K-1}-2, K^2\right)$. When $K\geq 5$, Theorem \ref{k/4} implies that when $n\geq 4K^2+5K+24+\frac{96}{K-4}$, the associated strongly regular graph satisfies Brouwer's Conjecture. In the next two sections, we improve this result when $K\in \{3,4\}$.

\subsection{Block graphs of Steiner triple systems}

A Steiner 2-$(n,3,1)$-design is called a Steiner triple system of order $n$ or STS(n). It is known that a STS(n) exists if and only if $n\equiv 1,3$ mod 6. If $n=7$, then the block graph of a STS(7) is the complete graph $K_7$. When $n\geq 9$, the block graph of a STS(n) is a strongly regular graph with parameters $\left(\frac{n(n-1)}{6}, \frac{3(n-3)}{2}, \frac{n+3}{2}, 9\right)$ and $2k-\lambda-2=\frac{6(n-3)}{2}- \frac{n+3}{2}-2=\frac{5n-25}{2}$. If $\delta$ and $\gamma$ are distinct points of a STS(n), we denote by $\{\delta,\gamma,*\}$ the block of the STS(n) containing both $\delta$ and $\gamma$.

In this section, we prove that any strongly regular graph that is the block graph of a Steiner triple system, satisfies Brouwer's Conjecture. 
\begin{theorem}\label{STS}
For every $n\geq 7, n\equiv 1,3\pmod{6}$, if $G$ is the block graph of a STS(n), then $\kappa_2(G)=2k-\lambda-2=\frac{5n-25}{2}$. Equality happens if and only if the disconnecting set is the neighborhood of an edge.
\end{theorem}
\begin{proof} We begin our proof with some small values of $n$. When $n=7$, the block graph of a STS(7) is the complete graph $K_7$ which is OK. When $n=9$, the block graph of a STS(9) has parameters $(12,9,6,9)$ which is OK by \cite[ Lemma 2.1]{CKK}. When $n=13$, the block graph of a STS(13) has parameters $(26,15,8,9)$ which is OK by \cite[Example 10.3]{CKK}.

Assume $n\geq 15$ for the rest of the proof. 

If $a=3$, then we have two cases: 
\begin{enumerate}

\item The set $A$ induces a triangle. 

We show that the vertices of $A$ have at least $3$ common neighbors outside $A$. If the three blocks in $A$ have non-empty intersection, we may assume they are $\{1,2,3\}$, $\{1,4,5\}$ and $\{1,6,7\}$. These vertices have at least $\frac{n-7}{2}\geq 4$ common neighbors. If the three blocks in $A$ have an empty intersection, we may assume they are $\{1,2,3\}$, $\{2,4,5\}$ and $\{3,5,6\}$. These vertices have at least three common neighbors: $\{1,5,*\}$, $\{3,4,*\}$ and $\{2,6,*\}$. In either situation, inclusion and exclusion implies that
$|S|\geq |N(A)|\geq 3(k-2)-3(\lambda-1)+3=3n-18>\frac{5n-25}{2}$ as $n>11$. 

\item The set $A$ induces a path of length $2$. 

We may assume that the vertices of $A$ are $\{1,2,3\}$, $\{2,4,5\}$ and $\{4,6,7\}$. These vertices have at least four common neighbors:  $\{1,4,*\}$, $\{3,4,*\}$, $\{2,6,*\}$ and $\{2,7,*\}$. By inclusion and exclusion, we get that $|S|\geq |N(A)|\geq k-2+2(k-1)-2\lambda-(\mu-1)+4=\frac{7n-49}{2}>\frac{5n-25}{2}$ as $n>12$.
\end{enumerate}

This finishes the proof of the case $a=3$. 

Assume that $a\geq 4$. When $n=15$,  Lemma \ref{Haemers} implies $|S|\geq ab\geq 4\times (31-|S|)$ which gives $|S|\geq 25=2k-\lambda-2$. We now show that $|S|>25$. If $|S|=25$, then $a+b=10$ and $a\leq b$ imply that $a\in \{4,5\}$. There are three cases to consider: 
\begin{enumerate}
\item The set $A$ induces a clique of size 4.

Let $x,y$ and $z$ be three distinct vertices of $A$. Inclusion and exclusion implies that $|S|\geq |N(A)|\geq |N(\{x,y,z\})|-1\geq 3(k-2)-3(\lambda-1)+3-1=3n-19>\frac{5n-25}{2}$ as $n>13$.

\item The set $A$ induces a clique of size 5. 

We claim that there are three vertices $x,y,z\in A$, whose blocks have non-empty intersection. Otherwise, we may assume $A$ contains three vertices of the form: $\{1,2,3\}$, $\{2,4,5\}$, $\{3,5,6\}$. Since $A$ has two other vertices, one may be $\{1,4,6\}$, and the other one must be one of the following forms: $\{1,5,*\}$, $\{3,4,*\}$ and $\{2,6,*\}$. In each situation, there will be three blocks with non-empty intersection. Assume that $x,y,z\in A$, and the three blocks corresponding to these vertices have non-empty intersection. Then they have at least $\frac{n-7}{2}\geq 4$ common neighbors. Inclusion and exclusion implies that $|S|\geq |N(A)|\geq |N(\{x,y,z\})|-2\geq 3(k-2)-3(\lambda-1)+4-2=3n-19>\frac{5n-25}{2}$ as $n>13$.

\item The set $A$ does not induces a clique. 

As $A$ induces a connected subgraph, we can find three vertices $x,y,z$ that induce a path of length 2. Inclusion and exclusion implies that $|S| \geq |N(A)|\geq |N(\{x,y,z\})|-2\geq k-2+2(k-1)-2\lambda-(\mu-1)+4-2=\frac{7n-53}{2}>\frac{5n-25}{2}$ as $n>14$.

\end{enumerate}

When $n\in \{19,21,25,27\}$, the parameters of $G$ satisfy Proposition \ref{2kl2} (with $c=4$), thus the disconnecting set is greater than $2k-\lambda-2$. 

Assume that $n\geq 31$ from now on. Let $p$ denote the number of points contained in the blocks corresponding to the vertices of $A$. We may assume that $p\leq n-p$. Otherwise, we can choose another component of $V\backslash S$ as our $A$. Obviously, $a\leq \frac{p(p-1)}{6}$. If $A$ does not induce a clique, let $x$ and $y$ be two non-adjacent vertices of $A$. Then
\begin{align*}
|S|&\geq |N(A)|\geq |N(\{x,y\})|-|A\setminus \{x,y\}|=2k-\mu-(a-2)\\
&=2k-\mu-a+2=3n-16-a.
\end{align*}
If $a<\frac{n-7}{2}$, then $|S|\geq 3n-16-a>\frac{5n-25}{2}$ and we are done. Otherwise, if $a\geq \frac{n-7}{2}$, then $\frac{p(p-1)}{6}\geq a\geq \frac{n-7}{2}$ which implies $p>\sqrt{3(n-7)}$. Now $B$ is spanned by at most $n-p$ points and thus, $|B|\leq \frac{(n-p)(n-p-1)}{6}$. Thus,
\begin{equation*}
|S|=|V|-|A|-|B|\geq \frac{n(n-1)}{6}-\frac{p(p-1)}{6}-\frac{(n-p)(n-p-1)}{6}=\frac{p(n-p)}{3}.
\end{equation*}
As $n\geq 31$, $p>\sqrt{3(n-7)}\geq \sqrt{72}$, so $p\geq 9$. Thus, $|S|\geq \frac{p(n-p)}{3}\geq 3(n-9)>\frac{5n-25}{2}$ since $n>29$. If $A$ induces a clique, then $k-2\lambda-1=\frac{n-17}{2}>0$ and Lemma \ref{Clique} imply that $|S|\geq |N(A)|>2k-\lambda-2$. This finishes our proof.
\end{proof}

\subsection{Block graphs of Steiner 2-$(n,4,1)$-designs}

A Steiner 2-$(n,4,1)$-design exists if and only if $n\equiv 1,4$ mod 12. When $n=13$, the block graph of such a design is the complete graph $K_{13}$. When $n\geq 16$, the block graph of a Steiner 2-$(n,4,1)$-design is a strongly regular graph with parameters $\left(\frac{n(n-1)}{12}, \frac{4(n-4)}{3}, \frac{n+20}{3}, 16\right)$ and $2k-\lambda-2=\frac{8(n-4)}{3}-\frac{n+20}{3}-2=\frac{7n-58}{3}$. 

In this section, we prove that any strongly regular graph that is the block graph of a Steiner 2-$(n,4,1)$-design, satisfies Brouwer's Conjecture. 
\begin{theorem}\label{SQS}
For $n\geq 13, n\equiv 1,4\pmod{12}$, if $G$ is the block graph of a 2-$(n,4,1)$-design, then $\kappa_2(G)=2k-\lambda-2=\frac{7n-58}{3}$. Equality happens if and only if the disconnecting set is the neighborhood of an edge.
\end{theorem}
\begin{proof}
When $n=13$, the block graph of a 2-$(13,4,1)$-design is the complete graph $K_{13}$ which is OK. When $n=16$, the block graph of a 2-$(16,4,1)$-design is strongly regular graph with parameters $(20,16,12,16)$. This is a complete multipartite graph which is OK by \cite[Lemma 2.1]{CKK}. 

Assume that $25\leq n\leq126$. Suppose first that $3\leq a\leq 5$. We have two possible cases:
\begin{enumerate}
\item The set $A$ contains three vertices $x, y$ and $z$ that form a triangle. We consider two subcases depending on whether or not the blocks of $x, y, z$ have nonempty intersection.  

\begin{enumerate}

\item The blocks corresponding to $x, y$ and $z$ have non-empty intersection. 

We may assume that $x=\{1,2,3,4\}, y=\{1,5,6,7\}$ and $z=\{1,8,9,10\}$. In this case, $x, y$ and $z$ have at least $\frac{n-10}{3}\geq 5$ common neighbors. 

\item The blocks corresponding to $x, y$ and $z$ have an empty intersection. 

If $\delta$ and $\gamma$ are distinct points of the Steiner 2-design, denote by $\{\delta,\gamma,*,*\}$ the block of the Steiner system containing $\delta$ and $\gamma$. We may assume that $x=\{1,2,3,4\}, y=\{2,5,6,7\}$ and $z=\{1,5,8,9\}$. In this case, $x, y$ and $z$ have at least 6 common neighbors: $\{1,6,*,*\}$,$\{1,7,*,*\}$, $\{2,8,*,*\}$, $\{2,9,*,*\}$, $\{3,5,*,*\}$, $\{4,5,*,*\}$.
\end{enumerate}
Inclusion and exclusion and $n\geq 25$ imply that
\begin{align*} 
|S|&\geq |N(A)| \geq |N(\{x,y,z\})|-(a-3) \\
           &\geq \left[3(k-2)-3(\lambda-1)+5\right]-2\\
           &=3n-36\geq\frac{7n-58}{3}.
\end{align*}
If equality holds above, then $n=25, a=5, b=6$ and $s=39$. However, Lemma \ref{Haemers} implies that
\begin{equation}\label{nabs}
|S|\geq \frac{4\times 5\times 6\times 16}{(15-16)^2+4\times (28-16)}\approx 39.1836
\end{equation}
which is a contradiction that finishes the proof of this subcase.

\item The subgraph induced by $A$ contains no triangles. 

Let $x\sim y\sim z$ denote an induced path of length $2$ that is contained in the subgraph induced by $A$. We may assume that $x=\{1,2,3,4\}, y=\{4,5,6,7\}$ and $z=\{7,8,9,10\}$. These vertices have at least 6 common neighbors: 
$\{4,8,*,*\}$, $\{4,9,*,*\}$, $\{4,10,*,*\}$, $\{1,7,*,*\}$, $\{2,7,*,*\}$, $\{3,7,*,*\}$. Inclusion and exclusion and $n\geq 25$ imply that
\begin{align*}
|S|&\geq |N(A)|\geq |N(\{x,y,z\})|-(a-3) \\
          &\geq \left[k-2+2(k-1)-2\lambda-(\mu-1)+6\right]-2\\
          &=\frac{10n-133}{3}\geq\frac{7n-58}{3}.
\end{align*}
If equality holds above, then $n=25, a=5, s=39$ and $b=6$. By the same argument as in \eqref{nabs}, we obtain again a contradiction.  
\end{enumerate}

Suppose now that $a\geq 6$ and $25\leq n\leq 107$. The hypothesis of Proposition \ref{2kl2} (with $c=6$) holds if $f(n)=-7n^3/27 + 326n^2/9 - 2848n/3 + 139168/27>0$. As the roots of the polynomial $f(n)$ are approximately $107.3212,  24.9760, 7.4171$, we obtain that $f(n)>0$ if $25\leq n\leq 107$. Thus, the disconnecting set is greater than $2k-\lambda-2$ in this case.

Suppose that $a=6$ and $108\leq n \leq 126$. If $A$ induces a clique, $k-2\lambda-1=\frac{2n-59}{3}>0$ and Lemma \ref{Clique} imply that $|S|\geq |N(A)|>2k-\lambda-2$. If $A$ does not induce a clique, by the same argument in case 2, we get that
\begin{align*}
|S|&\geq |N(A)|\geq |N(\{x,y,z\})|-(a-3) \\
          &\geq \left[k-2+2(k-1)-2\lambda-(\mu-1)+6\right]-3\\
          &=\frac{10n-136}{3}>\frac{7n-58}{3}.
\end{align*}

Suppose that $a\geq 7$ and $108\leq n \leq 126$. The hypothesis of Proposition \ref{2kl2} (with $c=7$) holds if $f(n)=-7n^3/27 + 374n^2/9 - 1104n + 150112/27>0$. As the roots of the polynomial $f(n)$ are approximately $128.4292,  25.2414, 6.6152$, we obtain that $f(n)>0$ if $108\leq n\leq 126$. Thus, the disconnecting set is greater than $2k-\lambda-2$. 

Assume now that $n\geq 127$. Let $p$ denote the number of points contained in the blocks corresponding to the vertices of $A$. We may assume that $p\leq n-p$. Otherwise, we can choose some other component of $V\backslash S$ as our $A$. Obviously, $a\leq \frac{p(p-1)}{12}$. If $A$ does not induce a clique, let $x$ and $y$ be two non-adjacent vertices of $A$. Then
\begin{align*}
|S|&\geq |N(A)|\geq |N(\{x,y\})|-|A\setminus \{x,y\}|=2k-\mu-(a-2)\\
&=2k-\mu-a+2=\frac{8n-74}{3}-a.
\end{align*}
If $a<\frac{n-16}{3}$, then $|S|\geq \frac{8n-74}{3}-a>\frac{7n-58}{3}$ and we are done. Otherwise, if $a\geq \frac{n-16}{3}$, then $\frac{p(p-1)}{12}\geq a\geq \frac{n-16}{3}$ which implies $p>\sqrt{2(n-16)}$. Now $B$ is spanned by at most $n-p$ points and thus, $|B|\leq \frac{(n-p)(n-p-1)}{12}$. This implies
\begin{align*}
|S|&=|V|-|A|-|B|\\
&\geq \frac{n(n-1)}{12}-\frac{p(p-1)}{12}-\frac{(n-p)(n-p-1)}{12}\\
&=\frac{p(n-p)}{6}\geq \frac{\sqrt{2(n-16)}(n-\sqrt{2(n-16)})}{6}.
\end{align*}
Now $\frac{\sqrt{2(n-16)}(n-\sqrt{2(n-16)})}{6}>\frac{7n-58}{3}$ is equivalent to 
$f(n)=n^3-144n^2+2368n-10952>0$ which is true when $n\geq 126$. Thus, $|S|>\frac{7n-58}{3}$ in this case. If $A$ induces a clique, $k-2\lambda-1=\frac{2n-59}{3}>0$ and Lemma \ref{Clique} imply that $|S|\geq |N(A)|>2k-\lambda-2$ which finishes our proof.
\end{proof}

\section{Brouwer's Conjecture for Latin square graphs}

An orthogonal array $OA(t,n)$ with parameters $t$ and $n$ is a $t\times n^2$ matrix with entries from the set $[n]=\{1,\dots,n\}$ such that the $n^2$ ordered pairs defined by any two distinct rows of the matrix are all distinct. It is well known that an orthogonal $OA(t,n)$ is equivalent to the existence of $t-2$ mutually orthogonal Latin squares (see \cite[Section 10.4]{GR}). In this paper, we use Godsil and Royle's notation $OA(t,n)$ from \cite[Section 10.4]{GR} to denote orthogonal arrays. Note that in other books such as Brouwer and Haemers \cite{BH} the orthogonal array $OA(t,n)$ is denoted by $OA(n,t)$. Given an orthogonal array $OA(t,n)$, one can define a graph $G$ as follows: The vertices of $G$ are the $n^2$ columns of the orthogonal array and two distinct columns are adjacent if they have the same entry in one coordinate position. The graph $G$ is an $(n^2, t(n-1), n-2+(t-1)(t-2), t(t-1))$-SRG. Any strongly regular graph with such parameters is called a Latin square graph (see \cite[Section 9.1.12]{BH}, \cite[Section 10.4]{GR} or \cite[Chapter 30]{VanLintWilson}). When $t=2$ and $n\neq 4$, such a graph must be the line graph of $K_{n,n}$ which is also the graph associated with an orthogonal array $OA(2,n)$ (see \cite[Problem 21F]{VanLintWilson}); this graph is OK by \cite[Section 8]{CKK}. When $t=2$ and $n=4$, there are two strongly regular graphs with parameters $(16,6,2,2)$, the line graph of $K_{4,4}$ and the Shrikhande graph (see \cite[p. 125]{BH}); they are both OK by \cite[Section 10]{CKK}.

The following theorem is the main result of this section and it shows that Latin square graphs with parameters $(n^2,t(n-1),n-2+(t-1)(t-2),t(t-1))$ satisfy Brouwer's Conjecture when $n\geq 2t\geq 6$. In particular, this will imply that strongly regular graphs obtained from orthogonal arrays $OA(t,n)$ satisfy Brouwer's Conjecture when $n\geq 2t\geq 6$.
\begin{theorem}\label{OA}
Let $t\geq 3$ be an integer. For any $n\geq 2t$, if $G$ is an $(n^2, t(n-1), n-2+(t-1)(t-2), t(t-1))$-SRG, then 
$$\kappa_2(G)=2k-\lambda-2=(2t-1)n-t^2+t-2.$$
The only disconnecting set of size $(2t-1)n-t^2+t-2$ are of the form $N(\{uv\})$, where $u$ and $v$ are two adjacent vertices in $G$.
\end{theorem}
\begin{proof} Recall that from our assumption on page 2 that $A$ is the vertex set of a component of $G\setminus S$ of minimum size, we have that $b=|B|\geq |A|=a$.

If $A$ induces a clique in $G$, then the inequality $k-2\lambda-1=(t-2)n-2t^2+5t-1\geq (t-2)(2t)-2t^2+5t-1=t-1>0$ and Lemma \ref{Clique} imply that $|S|>2k-\lambda-2$.

If $A$ does not induce a clique in $G$, then $A$ must contain an induced path of length $2$. Inclusion and exclusion implies that $a+s\geq 3k-2\lambda-\mu=(3t-2)n-3t^2+4t$. 

Assume by contradiction that $s\leq 2k-\lambda-2=(2t-1)n-t^2+t-2$. Then 
\begin{align*}
a&=(a+s)-s\geq (3k-2\lambda-\mu)-(2k-\lambda-2)=k-\lambda-\mu+2\\
&=(t-1)(n-2t+1)+3.
\end{align*}
Since $b\geq a$ and $a+b=v-s$, we get that $b\geq a\geq (t-1)(n-2t+1)+3$ and $a\leq (v-s)/2$. We write $ab=a(v-s-a)$ and we regard the expression $a(v-s-a)$ as a function of $a$. Let $f(x)=x(v-s-x)$ where $(t-1)(n-2t+1)+3\leq x\leq (v-s)/2$. Since $f'(x)=(v-s)-2x\geq 0$ for $x\leq (v-s)/2$, we deduce that $f(x)$ will increase on the interval $[(t-1)(n-2t+1)+2,(v-s)/2]$. Therefore, $f(x)$ will be minimum when $x=(t-1)(n-2t+1)+3$. Hence, we obtain that
\begin{align*}
ab&=a(v-s-a)=f(a)\geq f\left((t-1)(n-2t+1)+3\right)\\
&=[(t-1)(n-2t+1)+3][v-s-(t-1)(n-2t+1)-3]\\
&\geq [(t-1)(n-2t+1)+3][n^2-(2t-1)n+t^2-t+2-(t-1)(n-2t+1)-3]\\
&=[(t-1)(n-2t+1)+3][n^2-(3t-2)n+3t^2-4t],
\end{align*}
where the last inequality follows from $v-s\geq n^2-(2t-1)n+t^2-t+2$.

By Lemma \ref{Haemers}, we have
\begin{equation}\label{s}
s\geq \frac{4t(t-1)[(t-1)(n-2t+1)+3][n^2-(3t-2)n+3t^2-4t]}{n^2}.
\end{equation}
Note that  $n^2-(3t-2)n+3t^2-4t\geq \frac{n^2}{4}$. This is because for any fixed $t\geq 3$, the function $g(n)=3n^2-(12t-8)n+12t^2-16t$ is increasing in $[2t,\infty]$. Thus, $g(n)\geq g(2t)= 0$. Using this fact in inequality \eqref{s}, we obtain that
\begin{equation}\label{t}
s\geq t(t-1)[(t-1)(n-2t+1)+3].
\end{equation}
Let $f(n)=t(t-1)[(t-1)(n-2t+1)+3]-[(2t-1)n-t^2+t-2]$. When $t\geq 3$, $f(n)$ is 
increasing in $[2t, \infty]$. Thus, 
\begin{align*}
f(n)&\geq f(2t)=t(t-1)(t+2)-(3t^2-t-2)\\
&>t(t-1)(t+2)-3t^2\\
&=t(t^2-2t-2) >0.
\end{align*}
Combining with inequality \eqref{t}, we have $s>(2t-1)n-t^2+t-2=2k-\lambda-2$. This is a contradiction that finishes our proof.
\end{proof}

The following consequence of Theorem \ref{OA} extends Lemma 9.1 from \cite{CKK}. 
\begin{corollary}\label{t_34}
For $t\in \{3,4\}$ and any integer $n\geq t$, any strongly regular graph associated with an $OA(t,n)$ satisfies Brouwer's Conjecture.
\end{corollary}
\begin{proof}
For $t=3$, we need to see what happens when $n\in \{3,4,5\}$. If $n=3$, the corresponding graph is the complete three-partite graph $K_{3,3,3}$ which is OK by \cite[Lemma 2.1]{CKK}. If $n=4$, the corresponding graph is a $(16,9,4,6)$-SRG which is OK by \cite[Lemma 2.1]{CKK}. If $n=5$, the corresponding graph is a $(25,12,5,6)$-SRG which is OK by \cite[Lemma 2.5]{CKK}. 
 
For $t=4$, we need to check what happens when $n\in \{4,5,6,7\}$. If $n=4$, the corresponding graph is the complete  four-partite graph $K_{4,4,4,4}$ which is OK by \cite[Lemma 2.1]{CKK}. If $n=5$, the corresponding graph is a $(25,16,9,12)$-SRG is OK by \cite[Lemma 2.1]{CKK}. If $n=6$, as shown by Tarry (see also Stinson \cite{Stinson}) an orthogonal array $OA(4,6)$ does not exist. We remark here that McKay and Spence proved that there are exactly $32458$ strongly regular graphs with parameters $(36,20,10,12)$. If $n=7$, the corresponding graph is a $(49,24,11,12)$-SRG which is OK \cite[Lemma 2.5]{CKK}. 
\end{proof}

\section{The edge version of Brouwer's Conjecture}

In this section, we give a short proof of the edge version of Brouwer's Conjecture. We remark here that similar results were obtained by Hamidoune, Llad\'{o}, Serra and Tindell \cite{HLST} for some families of vertex-transitive graphs. 
\begin{theorem}
Let $G$ be a connected $(v,k,\lambda,\mu)$-SRG. If $A$ is a subset of vertices with $2\leq |A|\leq v/2$, then 
\begin{equation}\label{edgeBrouwer}
e(A,A^c)\geq 2k-2.
\end{equation}
Equality happens if and only if $A$ consists of two adjacent vertices or $A$ induces a triangle in $K_{2,2,2}$ or $A$ induces a triangle in the line graph of $K_{3,3}$. 
\end{theorem}
\begin{proof} If $k=3$, then $G$ is $K_{3,3}$ or the Petersen graph. If $G$ is $K_{3,3}$, the proof is immediate. If $G$ is the Petersen graph, and $A$ is a subset of vertices with $3\leq |A|\leq 4$, then the number of edges contained in $A$ is at most $|A|-1$ and therefore $e(A,A^c)\geq 3|A|-2(|A|-1)=|A|+2\geq 5$. If $|A|=5$, then the number of edges inside $A$ is at most $|A|=5$ and therefore $e(A,A^c)\geq 3|A|-2|A|=|A|=5$. If $k=4$, then $G$ is $K_{4,4}, K_{2,2,2}$ or the line graph of $K_{3,3}$. If $G$ is $K_{4,4}$, the proof is immediate. If $G$ is $K_{2,2,2}$, and $A$ is a subset of $3$ vertices inducing a triangle, then $e(A,A^c)=6$; in all the other cases, $e(A,A^c)>6$. If $G$ is the line graph of $K_{4,4}$ and $A$ is a subset of vertices with $|A|=3$, then $e(A,A^c)\geq 6$ with equality if and only if $A$ induces a clique of order $3$. If $|A|=4$, then the number of edges contained in $A$ is at most $4$ and therefore, $e(A,A^c)\geq 4|A|-8=8$.

Assume $k\geq 5$. If $3\leq |A|\leq k-2$, then $e(A,A^c)\geq |A|(k+1-|A|)\geq 3(k-2)>2k-2$ as $k\geq 5$. If $k-1\leq |A|\leq v/2$, then $e(A,A^c)\geq \frac{(k-\theta_2)a(v-a)}{v}\geq \frac{(k-\theta_2)(k-1)}{2}$ (see \cite[Corollary 4.8.4]{BH} or \cite{Mohar}). If $G$ is a conference graph of parameters $(4t+1,2t,t-1,t)$, then $k-\theta_2=\frac{4t+1-\sqrt{4t+1}}{2}>4$ for $t\geq 3$. These inequalities imply that $e(A,A^c)>2(k-1)$ which finishes the proof of this case. If $G$ is not a conference graph and $k-\theta_2>4$, then $e(A,A^c)>2k-2$ and we are done again. The only case left is when $G$ is not a conference graph and $k-\theta_2\leq 4$. In this case, the eigenvalues of $G$ are integers, $\theta_2\geq k-4$ and $\theta_v\leq -2$ as $G$ is not a complete graph. Since $k-1\geq k-\mu=\theta_2(-\theta_v)\geq 2k-8$, we get $5\leq k\leq 7$. If $\theta_v=-2$, then by Seidel's characterization of strongly regular graphs with minimum eigenvalue $-2$ (see \cite[Section 9.2]{BH} or \cite{S}), $G$ must have parameters $(16,6,2,2)$ and eigenvalue $6, 2$ and $-2$. In this case, $e(A,A^c)\geq \frac{(k-\theta_2)a(v-a)}{v}=\frac{a(16-a)}{4}>10$ for $k-1=5\leq a\leq 8=v/2$. If $\theta_v\leq -3$, then we obtain $k-1\geq k-\mu=\theta_2(-\theta_v)\geq 3k-12$ which implies $k=5$. In this case, the graph $G$ is the folded $5$-cube which has parameters $(16,5,0,2)$ and eigenvalues $5, 1$ and $-3$. As before, $e(A,A^c)\geq \frac{(k-\theta_2)a(v-a)}{v}=\frac{a(16-a)}{4}>8$ for $3\leq a\leq 8$. This finishes our proof.
\end{proof}

\section{Final Remarks}

In 1979, Neumaier \cite{N} classified strongly regular graphs with smallest eigenvalue $-m$, where $m\geq 2$ is a fixed integer, by showing that, with finitely many exceptions, such graphs are of one of the following types:
\begin{enumerate}[(a)]
\item Complete multipartite graphs with classes of size $m$, 
\item Latin square graphs with parameters $(n^2,m(n-1),n-2+(m-1)(m-2),m(m-1))$, 
\item Block graphs of Steiner $m$-systems.
\end{enumerate}

For any fixed integer $m\geq 3$, the graphs in case (a) satisfy Brouwer's Conjecture by \cite[Lemma 2.1]{CKK}. By our results in Sections 3 and 4, all, but finitely many strongly regular graphs of type (b) or (c), satisfy Brouwer's Conjecture. This means that there are finitely many strongly regular graphs with smallest eigenvalue $-m$ that might not satisfy Brouwer's Conjecture. When $m=2$, Cioab\u{a}, Kim and Koolen \cite{CKK} proved that among the strongly regular graphs with smallest eigenvalue $-2$, the only counterexamples of Brouwer's Conjecture are the triangular graphs $T(m)$, where $m\geq 6$. By Theorem \ref{STS}, Theorem \ref{SQS} and Corollary \ref{t_34}, we know that when $m=3,4$, the strongly regular graphs of type (b) and (c) satisfy Brouwer's Conjecture. 

We performed a computer search among feasible parameter sets of strongly regular graphs with smallest eigenvalue $-3$ or $-4$ for parameter sets $(v,k,\lambda,\mu)$ that do not satisfy the hypothesis of Proposition 2.4 in \cite{CKK}. When $m=3$, there are $321$ such parameters that are the same as the parameters of block graphs of Steiner triple systems and $66$ other parameter sets. Therefore, the status of Brouwer's Conjecture is established for all strongly regular graphs with minimum eigenvalue $-3$ with the exception of $387$ possible parameters. When $m=4$, there are $1532$ parameters that are the same as the parameters of block graphs of Steiner 2-$(n,4,1)$-designs and $232$ other parameters. Thus, the status of Brouwer's Conjecture is established for all strongly regular graphs with minimum eigenvalue $-4$ with the exception of $1764$ possible parameters.

\section*{Acknowledgments}
We are grateful to the referees for their excellent comments and suggestions.

\end{document}